\documentclass[12pt,oneside]{amsart}

\usepackage{amsfonts}
\usepackage{amsmath, amsthm, amssymb, ulem, amscd}

\def\crn#1#2{{\vcenter{\vbox{
        \hbox{\kern#2pt \vrule width.#2pt height#1pt
           }
          \hrule height.#2pt}}}}

\newcommand{\stopthm}{\hfill$\square$\medskip}

\newcommand{\pa}{\partial}
\newcommand{\Ric}{\operatorname{Ric}}

\newcommand{\R}{\mathbb R}
\newcommand{\N}{\mathbb N}

\newcommand{\ep}{\epsilon}

\newcommand{\ga}{\gamma}

\newcommand{\Wb}{\overline{W}}
\newcommand{\nb}{\bar{n}}

\newcommand{\ft}{\widetilde{f}}

\newcommand{\Dt}{\widetilde{\Delta}} 
\newcommand{\Ut}{\widetilde{U}}

\newcommand{\gt}{\widetilde{g}}
\newcommand{\cL}{\mathcal{L}}
\newcommand{\cLb}{\bar{\mathcal{L}}}

\newcommand{\cR}{\mathcal{R}}

\newcommand{\cM}{\mathcal{M}}

\newcommand{\cD}{\mathcal{D}}

\newcommand{\cMt}{\widetilde{\mathcal{M}}}

\theoremstyle{plain}
\newtheorem{theorem}{Theorem}[section]
\newtheorem{lemma}[theorem]{Lemma}
\newtheorem{proposition}[theorem]{Proposition}

\theoremstyle{definition}

\theoremstyle{remark}

\numberwithin{equation}{section}

\title[Juhl's Formulae]{Juhl's Formulae for GJMS Operators and
  $Q$-curvatures}   

\author{Charles Fefferman}
\address{Department of Mathematics, Princeton University\\
Princeton, NJ 08544}
\email{cf@math.princeton.edu} 

\author{C. Robin Graham}
\address{Department of Mathematics, University of Washington,
Box 354350\\
Seattle, WA 98195-4350}
\email{robin@math.washington.edu}

\begin{document}

\maketitle

\thispagestyle{empty}

\renewcommand{\thefootnote}{}
\footnotetext{Partially supported by NSF grants DMS 0901040 and DMS 
  0906035.}     
\renewcommand{\thefootnote}{1}

\section{Introduction}\label{intro}

GJMS operators and $Q$-curvatures are important objects in conformal 
geometry which have been studied intensely during the past decade.  In
\cite{J2}, \cite{J3}, building on previous work beginning with \cite{J1},  
Juhl has derived remarkable formulae for GJMS operators and  
$Q$-curvatures which reveal unexpected algebraic structure.  In this paper
we give direct proofs of Juhl's formulae starting from the original  
construction of \cite{GJMS}.   

Juhl's formulae are expressed in terms of quantities arising in the  
expansion of a Poincar\'e metric, or equivalently an ambient metric,
associated to a given pseudo-Riemannian metric.   
Let $g$ be a pseudo-Riemannian metric of signature $(p,q)$, $p+q=n\geq 3$,
on an $n$-dimensional manifold $M$.     
A Poincar\'e metric in normal form relative to $g$ is a metric $g_+$ on
$M\times (0,\epsilon)$ of the form
$$
g_+=r^{-2}\left(dr^2+h_r\right),
$$
where $h_r$ is a smooth 1-parameter family of metrics on $M$ satisfying
$h_0=g$, for which $\Ric(g_+)+ng_+=0$ in the following asymptotic sense.  
If $n$ is odd, then $\Ric(g_+)+ng_+=O(r^\infty)$, while if $n$ 
is even, then $\Ric(g_+)+ng_+= O(r^{n-2})$ and the tangential trace of 
$r^{2-n}\left(\Ric(g_+)+ng_+\right)$ vanishes at $r=0$.  Set 
$$
V(r)=\sqrt{\frac{\det h_r}{\det h_0}}
$$
and $W(r)=\sqrt{V(r)}$.  
Let $\delta$ denote the divergence operator on vector fields with respect
to $g$, given by $\delta \varphi= \nabla_i\varphi^i$.   
Define a 1-parameter family $\cM(r)$ of second order differential operators
on $M$ by 
\begin{equation}\label{Ddef}
\cM(r)=\delta (h_r^{-1}d) -U(r),
\end{equation}
where 
$$
U(r)=\frac{\left[\pa_r^2 
-(n-1)r^{-1}\pa_r+\delta (h_r^{-1}d)\right]W(r)}{W(r)}
$$
acts as a zeroth order term.  (We write $U(r)$ 
in the form given in v1 of \cite{J3}.  v2 of \cite{J3} expresses it in a  
different form; see Lemma 8.1 of v2.)   
Use $\cM(r)$ as a generating function for second order differential
operators 
$\cM_{2N}$ on $M$ defined for $N\geq 1$ (and $N\leq n/2$ if $n$ is even) by  
\begin{equation}\label{Mdef}
\cM(r)= \sum_{N\geq 1}\cM_{2N}\frac{1}{(N-1)!^2}  
\left(\frac{r^2}{4}\right)^{N-1}.  
\end{equation}
The $\cM_{2N}$ are natural scalar differential operators.  
Natural scalar invariants $W_{2N}$ are defined by
\begin{equation}\label{Wdef}
W(r)= 1+\sum_{N\geq 1} W_{2N}r^{2N}
\end{equation}
for $N\geq 1$ (and $N\leq n/2$ if $n$ is even).  

Juhl's formulae 
involve constants $n_I$, $m_I$ which are parametrized by ordered lists 
$I=(I_1,\ldots,I_r)$ of positive integers.  $I$ is referred to as a
composition of the sum $|I|=I_1+I_2+\cdots +I_r$.  
Sometimes compositions are written in the form   
$(I,a)$ singling out the last entry.  In this case the convention is  
that $I$ is allowed to be empty but $a>0$.   
The constants appearing in Juhl's formulae are:
\begin{equation}\label{nm}
\begin{split}
n_I&=(|I|-1)!^2\prod_{j=1}^r\frac{1}{(I_j-1)!^2}
\prod_{j=1}^{r-1}\frac{1}
{\left(\sum_{k=1}^jI_k\right)\left(\sum_{k=j+1}^rI_k\right)}\\
m_I&=(-1)^{r+1}|I|!(|I|-1)!\prod_{j=1}^r\frac{1}{I_j!(I_j-1)!}
\prod_{j=1}^{r-1}\frac{1}{I_j+I_{j+1}}.
\end{split}
\end{equation}
Empty products are always interpreted as 1.  Observe when $r=1$ that
$n_{(N)}=m_{(N)}=1$.  

Let $P_{2N}$ denote the GJMS operators, with sign convention determined by 
$P_{2N}=\Delta^N + \ldots$ with $\Delta = \delta(g^{-1}d)$.  These are
defined for all $N\geq 1$ for $n$ odd and for $1\leq N\leq n/2$ for $n$
even.  Iterated compositions of the $P_{2N}$ and the $\cM_{2N}$ are
denoted by $P_{2I}=P_{2I_1}\circ \cdots \circ P_{2I_r}$ and 
$\cM_{2I}=\cM_{2I_1}\circ \cdots \circ \cM_{2I_r}$.  

Juhl proves four formulae:  an explicit formula and a recursive
formula each for GJMS operators and for $Q$-curvatures.  All four formulae
are universal in the dimension.   

\begin{theorem}\label{GJMSexplicit}{\bf Explicit formula for GJMS
    operators.} 
For $N\geq 1$ (and $N\leq n/2$ if $n$ is even),
\begin{equation}\label{opexform}
P_{2N}=\sum_{|I|=N}n_I\cM_{2I}.
\end{equation}
\end{theorem}

\begin{theorem}\label{GJMSrecursive}{\bf Recursive formula for GJMS
    operators.} 
For $N\geq 1$ (and $N\leq n/2$ if $n$ is even),
\begin{equation}\label{oprecform}
P_{2N}=-\sum_{\stackrel{\scriptstyle{|I|=N}}{I\neq (N)}}m_IP_{2I}
+\cM_{2N}.  
\end{equation}
\end{theorem}

Clearly the explicit formula expresses $P_{2N}$ in terms
of the second order building blocks $\cM_{2M}$, $M\leq N$.  The recursive
formula expresses 
each $P_{2N}$ as a sum of compositions of lower order GJMS operators,
modulo the second order term $\cM_{2N}$.  For $N=1$ both formulae state
that $\cM_2=P_2$, the Yamabe operator.  
For $N=2$ the formulae express the Paneitz operator as
$P_4=\cM_2^2+\cM_4=P_2^2+\cM_4$.  The principal part of $\cM_{2N}$ for
$N>1$ involves curvature, and $\cM_{2N}=0$ for $N>1$ if $g$ is flat.  
Further discussion and specializations of the formulae may be found in
\cite{J3}.   

The $Q$-curvatures are defined in terms of the zeroth order terms of the
GJMS operators: 
\begin{equation}\label{Qdef}
P_{2N}(1)=(-1)^N\Big(\frac{n}{2}-N\Big)Q_{2N}. 
\end{equation}
$Q_{2N}$ is defined for all $N\geq 1$ if $n$ is odd and for $1\leq N\leq
n/2$ if $n$ is even.  For $n$ even, both sides vanish in the critical case
$N=n/2$ and $Q_n$ is defined by an analytic continuation.   

\begin{theorem}\label{Qexplicit}{\bf Explicit formula for Q-curvatures.}  
For $N\geq 1$ (and $N\leq n/2$ if $n$ is even),
\begin{equation}\label{Qexform}
(-1)^NQ_{2N}=\sum_{|(I,a)|=N}n_{(I,a)}a!(a-1)!2^{2a}\cM_{2I}(W_{2a}).
\end{equation}
\end{theorem}

\begin{theorem}\label{Qrecursive}{\bf Recursive formula for Q-curvatures.}
For $N\geq 1$ (and $N\leq n/2$ if $n$ is even),
\begin{equation}\label{Qrecform}
(-1)^NQ_{2N}=-\sum_{\stackrel{\scriptstyle{|(I,a)|=N}}{a< N}}
m_{(I,a)}(-1)^aP_{2I}(Q_{2a})+N!(N-1)!2^{2N}W_{2N}.
\end{equation}
\end{theorem}

The explicit formula expresses $Q_{2N}$ in terms of the
operators $\cM_{2M}$ and the coefficients $W_{2a}$.  The recursive formula 
expresses $Q_{2N}$ in terms of GJMS operators applied to $Q_{2a}$ with
$a<N$, modulo the multiple of $W_{2N}$.  Observe that
the factor $n/2-N$ in the definition which vanished in the critical case 
no longer appears. 
So the $Q$-curvature formulae do not follow immediately from the GJMS
operator formulae just by taking constant terms.  

If $g$ is Einstein or locally conformally flat, then 
there is an invariantly defined Poincar\'e metric to infinite order
also if $n$ is even.  It can be written explicitly; see \cite{FG}.  In
these cases, $P_{2N}$ and $Q_{2N}$ are invariantly defined for all $N\geq
1$ also for $n$ even.  For such $g$, Juhl's formulae and our 
proofs are valid for all $N$.     

The GJMS operators are known to be self-adjoint.  This is exhibited by the
formulae \eqref{opexform} and \eqref{oprecform}, since 
$\cM(r)$ is evidently self-adjoint with respect to $g$ for each $r$ so that
the $\cM_{2N}$ are all self-adjoint, and since $n_I=n_{I^{-1}}$ and
$m_I=m_{I^{-1}}$ where $I^{-1}=(I_r,\ldots, I_1)$.  However, the 
self-adjointness is not obvious from the original GJMS construction.  If
one desires to understand Juhl's formulae in terms of the original 
construction, it is reasonable to start by asking the modest question 
of how to see the self-adjointness from that derivation.  It turns out that   
understanding this is the key to unlocking the mysteries of Juhl's
formulae. 

The operators $P_{2N}$ were derived in \cite{GJMS} via 
the Laplacian of the ambient metric associated to $g$.  In normal form,
this is the metric  
\begin{equation}\label{ambmetric}
\gt=2\rho dt^2 +2tdtd\rho +t^2g_\rho
\end{equation}
on $\R_+\times M\times (-\ep,\ep)$, where $t\in \R_+$, $\rho \in
(-\ep,\ep)$, and $g_\rho=h_r$ with $\rho=-r^2/2$.  The asymptotic vanishing 
of $\Ric(g_+)+ng_+$ at $r=0$ translates into asymptotic vanishing of
$\Ric(\gt)$ at $\rho=0$.  See \cite{FG} for details.  If $f\in
C^\infty(M)$ and $\ft\in 
C^\infty(M\times (-\ep,\ep))$ satisfies $\ft(x,0)=f$, then the GJMS
definition is   
\begin{equation}\label{Pdef}
P_{2N}f=\Dt^N(t^{N-n/2}\ft)|_{\rho=0,t=1}
\end{equation}
where $\Dt$ denotes the Laplacian in the metric $\gt$.  
The right-hand side is shown to be independent of the choice of $\ft$
extending $f$.  

It is straightforward to calculate the Laplacian $\Dt$ of a metric of the 
form \eqref{ambmetric}.  Evidently there is a term  
involving the Laplacian $\Delta_{g_\rho}$ in the metric $g_\rho$ for fixed
$\rho$ acting in the $M$ factor (see \eqref{Dfirst} below).  This term is
not self-adjoint with    
respect to $g=g_0$, so $P_{2N}$ obtained by iterating $\Dt$ and restricting 
to $\rho=0$ does not appear to be self-adjoint either.  This is the reason
that self-adjointness of the $P_{2N}$ is not apparent from this 
construction.  However, if we set 
\begin{equation}\label{vt}
v(\rho)=\sqrt{\frac{\det g_\rho}{\det g_0}}
\end{equation}
(so that $v(\rho)=V(r)$ with $\rho=-r^2/2$), then multiplying
the volume form for $g$ by $v(\rho)$ gives the volume form for
$g_\rho$.  It follows that for each $\rho$ the operator
$v(\rho)\Delta_{g_{\rho}}$   
is self-adjoint with respect to $g$.  Pre- and post-composing a
self-adjoint operator with multiplication by a smooth real function gives 
another self-adjoint operator.  Therefore the operator 
$v^{1/2}\circ\Delta_{g_{\rho}}\circ v^{-1/2}$ is also self-adjoint with 
respect to $g$. This motivates consideration of    
$$
\Dt_v:=v^{1/2}\circ\Dt\circ v^{-1/2},
$$
as we are guaranteed that the operator acting along $M$ when $\Dt_v$ is 
written out will be self-adjoint with respect to 
$g$.  Moreover, $\Dt_v^N=v^{1/2}\circ\Dt^N\circ v^{-1/2}$ since the middle 
factors of $v^{\pm 1/2}$ cancel.  The pre- and post-multiplications by
$v^{\pm 1/2}$ affect neither the extension property nor 
the restriction back to $\rho=0$ since $v=1$  
at $\rho =0$.  Hence \eqref{Pdef} can be rewritten as  
\begin{equation}\label{Pdefnew}
P_{2N}f=\Dt_v^N(t^{N-n/2}\ft)|_{\rho=0,t=1}. 
\end{equation}
Now a direct calculation which we carry out in \S\ref{Gsect} shows that 
\begin{equation}\label{Dvformula}
\Dt_v(t^\gamma\ft)=t^{\gamma-2}\left[-2\rho\pa_\rho^2+(2\gamma+n-2)\pa_\rho 
+\cMt(\rho)\right]\ft,
\end{equation}
where $\cMt(\rho)=\cM(r)$, $\rho = -r^2/2$.  This is the key identity.  It 
explains the previously mysterious appearance of both $W=\sqrt{V}$ and 
the generating function $\cM(r)$ in Juhl's theory.  Since   
$\left[-2\rho\pa_\rho^2+(2\gamma+n-2)\pa_\rho\right]\rho^k
=c_{k,\gamma,n}\rho^{k-1}$
for constants $c_{k,\gamma,n}$, upon choosing $\ft$ to be independent of
$\rho$ we see that iterating \eqref{Dvformula} and restricting to
$\rho=0$, $t=1$ gives a formula for
$P_{2N}$ as a linear combination of compositions of the Taylor coefficients
of $\cMt(\rho)$, i.e. of the $\cM_{2I}$.  Showing that the coefficients in 
the linear combination are the $n_I$ reduces to a (rather nontrivial) 
combinatorial identity which 
we derive in \S\ref{combid}.  This proves Theorem~\ref{GJMSexplicit}.  
Theorem~\ref{Qexplicit} reduces to an equivalent combinatorial identity
upon calculating $P_{2N}1$ using \eqref{Pdefnew}, \eqref{Dvformula} and
taking the extension $\ft$ to be $v^{1/2}$ rather than $1$.  This     
reduction is included in \S\ref{Gsect} and the proof of the relevant 
combinatorial identity in \S\ref{combid}.    

Theorem~\ref{GJMSrecursive} can be derived from Theorem~\ref{GJMSexplicit} 
by inverting \eqref{opexform}, viewed as a formal transformation law from
the $\cM_{2N}$ to the $P_{2N}$.  A proof in the opposite direction due to 
Krattenthaler 
was presented in \S 2 of \cite{J3} and immediately implies the 
direction we need here.  Likewise, Theorem~\ref{Qrecursive} follows from
Theorem~\ref{Qexplicit} upon inverting \eqref{Qexform}, viewed as a formal 
transformation from the $W_{2N}$ to the $Q_{2N}$.  In \S\ref{recursive} we
review Krattenthaler's proof of the inversion for the operators following
the presentation in \cite{J3} and then
present the similar but more complicated proof for the $Q$-curvatures.  

It is also possible to prove both the explicit and recursive formulae for
$Q$-curvatures by taking 
the constant term in the corresponding formula for the GJMS operators and  
then rewriting by 
deriving and substituting expressions for the $\cM_{2N}(1)$.  This approach 
is closely related to arguments in \cite{J3}, where the scalar invariants 
$\cM_{2N}(1)$ play a prominent role.  

We still find these formulae to be quite astonishing. Juhl deserves great
credit for their discovery as subtle consequences of the recursive
structure of his residue families.  Even though
we now see that this theory of residue families and their factorization
identities is not required for their proofs, this theory, linking ideas
from conformal geometry, representation theory and spectral theory, appears  
deep and fascinating and deserves further exploration.

\section{Explicit Formulae}\label{Gsect} 

In this section we give the details of the argument outlined in the
introduction which reduces Theorem~\ref{GJMSexplicit} to a combinatorial 
identity and then show how similar reasoning reduces
Theorem~\ref{Qexplicit} to an equivalent combinatorial identity.  

The first task is to establish \eqref{Dvformula} by direct calculation.  
The inverse of the metric \eqref{ambmetric} is
$$
\gt^{IJ}=
\begin{pmatrix}
0&0&t^{-1}\\
0&t^{-2}g^{ij}_\rho&0\\
t^{-1}&0&-2\rho t^{-2}
\end{pmatrix}
$$
and $\sqrt{|\det \gt|}=t^{n+1}\sqrt{|\det g_\rho|}$.  Using this in   
$$
\Dt = \frac{1}{\sqrt{|\det \gt|}}\pa_I\left(\gt^{IJ}
\sqrt{|\det \gt|}\pa_J\right)
$$
gives
\begin{equation}\label{Dfirst}
\Dt(t^\ga\varphi)=t^{\ga-2}\left[-2\rho \varphi'' +(2\ga+n-2-2\rho 
v'/v )\,\varphi'
+(\Delta_{g_\rho}+\ga v'/v)\varphi \right]
\end{equation}
(cf. (3.5) of \cite{GJMS}).    
Here $v$ is given by \eqref{vt}, $'$ denotes $\pa_\rho$, and $\varphi$ is 
independent of $t$.  Set $w=v^{1/2}$ and 
$\varphi=w^{-1}\psi$.  Then $v'/v=2w'/w$ and  
\[
\begin{split}
\varphi' &= w^{-1}\psi' -w^{-2}w'\psi\\
\varphi'' &= w^{-1}\psi'' -2w^{-2}w'\psi'+(2w^{-3}w'^2-w^{-2}w'')\psi.
\end{split}
\]
Substituting and simplifying gives
\begin{equation}\label{firstpart}
\begin{split}
w\big[-2\rho \varphi'' +&(2\ga+n-2-2\rho 
v'/v )\varphi'+
\ga (v'/v)\varphi\big]\\
=&-2\rho\psi''+(2\ga +n-2)\psi'
+w^{-1}\big[2\rho w''-(n-2)w'\big]\psi. 
\end{split}
\end{equation}
For the remaining term in \eqref{Dfirst} we have
\begin{lemma}\label{lapla}
\begin{equation}\label{secondpart}
w\circ \Delta_{g_\rho}\circ w^{-1}
=\delta(g_\rho^{-1}d) -w^{-1}\delta(g_\rho^{-1}dw). 
\end{equation}
(Recall that $\delta$ denotes the divergence operator with respect to
$g=g_0$.)  The second term on the right-hand side acts as a zeroth order
operator.   
\end{lemma}
\begin{proof}
For fixed $\rho$ it is clear that 
$w\circ \Delta_{g_\rho}\circ w^{-1}$ and $\delta(g_\rho^{-1}d)$ are 
second order differential operators whose principal parts agree.   
We observed in the introduction that the first is self-adjoint 
with respect to $g$, and clearly this is the case for the second.  So their
difference is zeroth order.  Evaluating on $w$ identifies the zeroth order
term. 
\end{proof}

Multiplying \eqref{Dfirst} by $w$ and substituting \eqref{firstpart} and
\eqref{secondpart} yields 
\begin{equation}\label{almost}
w\Dt(t^\ga w^{-1}\psi)
=t^{\ga-2}\left[-2\rho\psi''+(2\ga+n-2)\psi'
+\left(\delta(g_\rho^{-1} d)-\Ut(\rho)\right)\psi \right]
\end{equation}
where 
$$
\Ut(\rho)=\frac{\left[-2\rho\pa_\rho^2+(n-2)\pa_\rho 
+\delta(g_\rho^{-1}d)\right]w(\rho)}{w(\rho)}.
$$
The chain rule with $\rho=-r^2/2$ shows that $\Ut(\rho)=U(r)$ so that 
$\delta(g_\rho^{-1} d)-\Ut(\rho)=\cMt(\rho)$.  Hence \eqref{almost} becomes  
\eqref{Dvformula}.  This completes the derivation of \eqref{Dvformula}.  

Set 
$$
\cR_k=-2\rho\pa_\rho^2+2k\pa_\rho +\cMt(\rho)
$$
and note that \eqref{Mdef} becomes
$$
\cMt(\rho)=\sum_{N\geq 1}\cM_{2N}\frac{1}{(N-1)!^2}  
\left(-\frac{\rho}{2}\right)^{N-1}.  
$$
Iterating \eqref{Dvformula} gives
$$
\Dt_v^N(t^{N-n/2}\ft)=
t^{-N-n/2}\cR_{1-N}\cR_{3-N}\cdots \cR_{N-3}\cR_{N-1}\ft, 
$$
so we deduce that $\cR_{1-N}\cR_{3-N}\cdots 
\cR_{N-3}\cR_{N-1}\ft|_{\rho =0}$ 
depends only on $\ft|_{\rho =0}$.  Taking $\ft$ to be independent of 
$\rho$, it follows upon expanding the right-hand side that 
$\cR_{1-N}\cR_{3-N}\cdots \cR_{N-3}\cR_{N-1}|_{\rho =0}$ is a linear
combination of the compositions $\cM_{2I}$.    
In the next section we will prove the combinatorial identity  
\begin{equation}\label{beforechange}
\cR_{1-N}\cR_{3-N}\cdots \cR_{N-3}\cR_{N-1}|_{\rho =0} =
\sum_{|I|=N}n_I\cM_{2I}  
\end{equation}
which identifies the constants in the linear combination. 
Theorem~\ref{GJMSexplicit} then follows via \eqref{Pdefnew}.   

We next show that Theorem~\ref{Qexplicit}, the explicit formula for
$Q$-curvatures, reduces to a similar combinatorial identity which we will
see in the next section is equivalent to \eqref{beforechange}.  
By definition we have $(-1)^N(n/2-N)Q_{2N}=P_{2N}1$.  Use \eqref{Pdefnew}
to calculate $P_{2N}1$, taking $\ft=v^{1/2}$ to be 
the extension of $f=1$.  Thus
$$
(-1)^N(n/2-N)Q_{2N}=\Dt_v^{N-1}\big(v^{1/2}\Dt(t^{N-n/2})\big)|_{\rho=0,t=1}.
$$
Equation \eqref{Dfirst} gives
$$
\Dt(t^{N-n/2})=t^{N-n/2-2}(N-n/2)v'/v=2t^{N-n/2-2}(N-n/2)w'/w.
$$
The factors of $(N-n/2)$ cancel, and it follows that
\begin{equation}\label{N-1}
(-1)^NQ_{2N}=-2\Dt_v^{N-1}(t^{N-n/2-2}w')|_{\rho=0,t=1}.
\end{equation}
Iterating \eqref{Dvformula} gives 
\begin{equation}\label{QR}
(-1)^NQ_{2N}=-2\cR_{1-N}\cR_{3-N}\cdots \cR_{N-3}(w')|_{\rho=0}.
\end{equation}
Now $w=1+\sum_{a\geq 1}W_{2a}(-2\rho)^a$, so
\begin{equation}\label{w'}
w'=\sum_{a\geq 1}a(-2)^aW_{2a}\rho^{a-1}.
\end{equation}
As will be shown in the next section, the following is equivalent to 
\eqref{beforechange}.
\begin{proposition}\label{modrecurcons}
Let $1\leq a\leq N$ and let $f$ be a function on $M$ (i.e. independent of 
$\rho$).  Then
\begin{equation}\label{Rform}
\cR_{1-N}\cR_{3-N}\cdots \cR_{N-3}(f\rho^{a-1})|_{\rho=0}
=\sum_{|I|=N-a}n_{(I,a)}(a-1)!^2(-2)^{a-1}\cM_{2I}(f).
\end{equation}
\end{proposition}

\noindent
Substituting \eqref{w'} into \eqref{QR} and applying \eqref{Rform} termwise
gives 
$$
(-1)^NQ_{2N}=\sum_{(I,a)=N}n_{(I,a)}a!(a-1)!2^{2a}\cM_{2I}(W_{2a}),
$$
which is the explicit formula for $Q_{2N}$.    

For $n$ even, the above argument applies also for the critical case $N=n/2$
since $Q_n$ 
is defined by removing the factor of $n/2-N$.  The critical case may also
be deduced without this argument of analytic continuation in the
dimension by using the realization
$$
(-1)^{n/2}Q_n= -\Dt^{n/2}(\log t)|_{\rho=0,t=1}
$$
derived in \cite{FH}.  Namely, write
$$
\Dt^{n/2}(\log t)|_{\rho=0,t=1}=\Dt_v^{n/2-1}\big(w\Dt(\log
t)\big)|_{\rho=0,t=1}. 
$$
Direct calculation gives $w\Dt(\log t)=2t^{-2}w'$.  So we recover
\eqref{N-1} and the argument proceeds as above.

\section{Combinatorial Identities}\label{combid}

In this section we derive the combinatorial identities
\eqref{beforechange} and \eqref{Rform} 
to which Theorems~\ref{GJMSexplicit} and \ref{Qexplicit} were reduced
above.  Begin with \eqref{beforechange}.  
First change variables:  set 
\begin{equation}\label{cov}
s=-\frac{\rho}{2},\qquad x_N=\frac{\cM_{2N}}{(N-1)!^2},\qquad  
X(s)=\cMt(\rho)=\sum_{N=0}^\infty x_{N+1}s^N.
\end{equation}
As far as this identity is concerned, 
$x_1$, $x_2,\ldots$ can simply be regarded as noncommuting variables, all
of which commute with $s$.  In the new variables, the $\cR_k$ become the 
differential operators
$$
\cL_{k}=s\frac{d^2}{ds^2}-k \frac{d}{ds} +X(s), 
$$
where $X(s)$ acts as a zeroth order multiplication operator.  We only have
to verify the constant term in $\pa_\rho$ of \eqref{beforechange}, which
becomes 
\begin{theorem}\label{combthm}
Let $N\geq 1$.  Then
\begin{equation}\label{comb}
\cL_{1-N}\cL_{3-N}\cdots \cL_{N-3}\cL_{N-1}1|_{s=0}  
=\sum_{|I|=N}\nb_I\,\,x_{I_1}x_{I_2}\cdots x_{I_r}, 
\end{equation}
where
$$
\nb_I=\frac{(N-1)!^2}{\prod_{k=1}^{r-1}
\Big(\sum_{j=1}^kI_j\Big)\Big(\sum_{j=k+1}^rI_j\Big)}.
$$
\end{theorem}

Set $\cLb_j=\cL_{N+1-2j}$ so that 
$\cL_{1-N}\cL_{3-N}\cdots \cL_{N-3}\cL_{N-1}
=\cLb_N\cLb_{N-1}\cdots \cLb_2\cLb_1$.  
Since $\nb_I=\nb_{I^{-1}}$, \eqref{comb} can be rewritten as
\begin{equation}\label{reduced}
\cLb_N\cLb_{N-1}\cdots \cLb_2\cLb_11|_{s=0}  
=\sum_{|I|=N}\nb_I\,\,x_{I_r}x_{I_{r-1}}\cdots x_{I_1}. 
\end{equation}
Fix positive integers $I_1,\ldots, I_r$, where $r\geq 1$.
We will prove \eqref{reduced} by verifying the 
coefficient of $x_{I_r}x_{I_{r-1}}\cdots x_{I_1}$ in 
$\cLb_N\cLb_{N-1}\cdots \cLb_2\cLb_11|_{s=0}$
for each choice of $I_1,\ldots,I_r$.  

For $1\leq l\leq r$, set 
$$
m_l=I_1+I_2+\cdots + I_l
$$
so that $1\leq m_1<m_2<\cdots <m_{r-1}<m_r$.  (The $m_I$ from the previous
sections will not appear in this section.)     
Consider the calculation of $\cLb_N\cLb_{N-1}\cdots \cLb_2\cLb_11$ by
successive multiplication from the left.  For $1\leq j\leq N$,  
$\cLb_j\cLb_{j-1}\cdots \cLb_1 1$ is a formal power series in $s$ whose
coefficients are polynomials in the $x$'s.   
The only monomials in the $x$'s appearing in   
$\cLb_j\cLb_{j-1}\cdots \cLb_1 1$ which can ultimately contribute to the
coefficient of $x_{I_r}x_{I_{r-1}}\cdots x_{I_1}$ in
$\cLb_N\cLb_{N-1}\cdots \cLb_2\cLb_11$ are 
of the form $x_{I_l}x_{I_{l-1}}\cdots x_{I_1}$ for some $l$, $1\leq l\leq
r$.  The term $sd^2/ds^2- (N+1-2k)d/ds$ in one of the factors $\cLb_k$
just reduces the power of $s$ by 1 and multiplies by a constant.  The term
$X(s)$ is linear in the $x$'s.  So in order for a monomial
$x_{I_l}x_{I_{l-1}}\cdots x_{I_1}$ to 
appear in the expansion of $\cLb_j\cLb_{j-1}\cdots \cLb_1 1$, it must be
that the zeroth order term $X(s)$ has contributed in exactly $l$ of these
$\cLb_k$.  Thus the differentiation in $s$ terms have contributed in
exactly $j-l$ of the $\cLb_k$.  It follows that the power of $s$
multiplying $x_{I_l}x_{I_{l-1}}\cdots x_{I_1}$ is 
$s^{m_l-l-(j-l)}=s^{m_l-j}$.  Hence we have 
\begin{equation}\label{cs}
\cLb_j\cLb_{j-1}\cdots \cLb_1 1 
=\sum_{l=1}^{\text{min}(j,r)}c_{j,l}\;x_{I_l}x_{I_{l-1}}\cdots
x_{I_1}s^{m_l-j} 
+\ldots
\end{equation}
for some constants $c_{j,l}$, where $\ldots$ indicates terms involving
monomials in the $x$'s which cannot contribute in the end.  The $c_{j,l}$
are defined for $1\leq j\leq N$, $1\leq l\leq \text{min}(j,r)$, and 
we have $c_{1,1}=1$ and $c_{j,l}=0$ if $m_l<j\leq N$.    

{From} \eqref{cs} it follows first that the coefficient of 
$x_{I_r}\cdots x_{I_1}$ in $\cLb_N\cLb_{N-1}\cdots
\cLb_2\cLb_11|_{s=0}$ is zero unless $|I|=N$.  In fact, taking $j=N$, the
term $x_{I_r}\cdots x_{I_1}$ on the right-hand side is multiplied by
$s^{m_r-N}$.  This vanishes at $s=0$ unless $m_r=N$, i.e. $|I|=N$.  
Theorem~\ref{combthm} therefore reduces to the statement that
$c_{N,r}=\nb_I$ if $|I|=N$.  We assume henceforth that $|I|=N$,
i.e. $m_r=N$.   

Extend the definition of the $c_{j,l}$ to $0\leq j\leq N$, 
$0\leq l\leq r$ by defining $c_{0,0}=1$ and $c_{j,l}=0$ if 
$0\leq j<l\leq r$ or if $l=0$ and $1\leq j\leq N$.  
We claim that these constants satisfy the recursion relation:
\begin{equation}\label{recursion}
c_{j+1,l} = - (m_l-j)(N-m_l-j)c_{j,l} + c_{j,l-1} 
\end{equation}
for $0\leq j\leq N-1$, $1\leq l\leq r$.  For 
$1\leq j\leq N-1$ and $1\leq l\leq \text{min}(j+1,r)$ this follows by
applying $\cLb_{j+1}$ to \eqref{cs}.  For $j=0$, $l=1$ both sides are 1,
and for all the other values both sides vanish.  Now extend the
definition of the $c_{j,l}$ to $j>N$, $0\leq l\leq r$ by setting
$c_{j,0}=0$ for $j>N$ and by requiring that
\eqref{recursion} hold for $j\geq N$, $1\leq l\leq r$.  The resulting
$c_{j,l}$ are defined for $j\geq 0$, $0\leq l\leq r$, and \eqref{recursion}
holds for $j\geq 0$, $1\leq l\leq r$.   

Define generating functions
$$
F_l(y)=\sum_{j=0}^\infty \frac{c_{j,l}}{(j!)^2}y^j,\qquad\quad 
0\leq l\leq r. 
$$
The definitions of the $c_{j,0}$ and $c_{0,l}$ show that  
\begin{equation}\label{Finit}
F_0=1\qquad\quad \text{and}\qquad\quad F_l(0)=0, \quad 1\leq l\leq r.
\end{equation}
The recursion
\eqref{recursion} turns into a differential equation relating $F_l$ and
$F_{l-1}$.  For fixed positive integral $N$ as above, define ordinary
differential operators  
$$
\cD_m= y(1+y)\frac{d^2}{dy^2}+[1-(N-1)y]
\frac{d}{dy}+m(N-m).
$$
\begin{lemma}\label{diffeq}
Let
\begin{equation}\label{expand}
u=\sum_{j=0}^\infty\frac{u_j}{(j!)^2}y^j, \qquad
f=\sum_{j=0}^\infty\frac{f_j}{(j!)^2}y^j
\end{equation}
be formal power series.  Then $\cD_mu=f$ if and only if  
\begin{equation}\label{recur}
u_{j+1}=-(m-j)(N-m-j)u_j + f_j,\qquad\qquad j\geq 0.
\end{equation}
\end{lemma}

\noindent
The proof is to substitute the expansions into the equation and to compare 
coefficients of like powers of $y$.  
Comparing \eqref{recursion} and \eqref{recur} then gives immediately  
\begin{equation}\label{DF}
\cD_{m_l}F_l = F_{l-1}, \qquad 1\leq l\leq r.
\end{equation}

Now $\cD_m$ has a regular singularity at $y=0$ with indicial root 0 of  
multiplicity 2.  By general Frobenius theory or just by staring at
\eqref{recur}, there exists a unique formal 
power series solution of $\cD_mu=0$ with $u(0)=1$.  Also, for any formal
power series $f$ there exists a unique formal power series solution $u$ to 
$\cD_mu=f$ with $u(0)=0$.  In particular, \eqref{Finit} and \eqref{DF}
together characterize the functions $F_l$.  Combining the solutions of the
homogeneous and inhomogeneous problems shows that for any $f$ there is a 
unique solution $u$ to $\cD_mu=f$ with $u(0)$ any prescribed value.

Since the $y^N$ coefficient of $F_r(y)$ is $c_{N,r}/(N!)^2$, the above
considerations show that the statement $c_{N,r}=\nb_I$ to which 
Theorem~\ref{combthm} reduced is a consequence of the following. 
\begin{proposition}\label{recusolve}
Let $r\geq 1$ and $1\leq m_1<m_2<\cdots< m_r=N$.  Define formal power series
$F_l(y)$ for $0\leq l\leq r$ by \eqref{Finit} and \eqref{DF}.
Then $F_r$ is a polynomial of degree $=N$ and its $y^N$ coefficient is  
$$
\left[N^2 \prod_{l=1}^{r-1}m_l(N-m_l)\right]^{-1}.
$$
\end{proposition}

\noindent
{\it Remarks.}  It follows easily from the discussion below (or from the
definition of the $c_{j,l}$) that $F_l$ is a  
polynomial of degree $\leq m_l$.  For $l<r$ it often happens that $\deg
F_l<m_l$.  It is easily seen from the definition of the $c_{j,l}$ (or
from \eqref{recursion}) that the lowest power of $y$ occuring in $F_l$ with
nonzero coefficient is $y^l$, and its coefficient is 1.   

We prove Proposition~\ref{recusolve} by expressing the $F_l(y)$ in terms 
of special solutions of the differential equations.    
Let $P_m$ denote the formal power series defined by
$$
\cD_mP_m=0,\qquad P_m(0)=1.
$$
Then $P_m=P_{N-m}$ since $\cD_m= \cD_{N-m}$.  Clearly $P_0(y)=1$.  
We claim that if $m$ is an integer satisfying $0\leq m\leq N$, then 
$P_m$ is a polynomial of degree $=\min(m,N-m)$.  This is clear from
\eqref{recur} with $f=0$ since the multiplicative factor first vanishes
when $j=\min(m,N-m)$.  Up to a simple linear     
change of independent variable and overall multiplicative factor, the $P_m$
are particular instances of Jacobi polynomials.

Next observe that the same reasoning applies if $f$ is a polynomial of 
degree $<\min(m,N-m)$: the unique solution $u$ with $u(0)$ any prescribed 
value is a 
polynomial of degree $\leq \min(m,N-m)$.  The multiplicative factor
$(m-j)(N-m-j)$ 
also vanishes for $j=\max(m,N-m)$.  Again the same reasoning shows that if
$f$ is a polynomial   
of degree $<\max(m,N-m)$, then $u$ is a polynomial of degree $\leq
\max(m,N-m)$.  In particular, if $m\neq N/2$ the conditions   
$$
\cD_mQ_m = P_m,\qquad Q_m(0)=0
$$
uniquely determine a polynomial $Q_m$ of degree 
$\leq \max(m,N-m)$.  Again $Q_m=Q_{N-m}$.  In the special case $m=0$, we
have  
\begin{lemma}\label{Qcoeff}
The $y^N$ coefficient of $Q_0$ is $N^{-2}$.   
\end{lemma}
\begin{proof}
We have $P_0=1$.  So \eqref{recur} with $j=0$ and $u_0=0$ gives $u_1=1$.
Setting $m=0$ and iterating \eqref{recur} for higher $j$ gives
$$
u_j = (j-1)!(N-j+1)(N-j+2)\cdots(N-1).
$$  Hence $u_N=(N-1)!^2$.  The result now follows from \eqref{expand}.
\end{proof}

\noindent
{\it Proof of Proposition~\ref{recusolve}.}
Begin by observing that the definition of the $F_l$ and the conclusion both
remain unchanged if any $m_l$ is replaced by $N-m_l$.  We use this
observation to redefine some of the $m_l$.  Namely, if $1\leq l\leq r-1$
and $m_l$ satisfies the two conditions that $m_l>N/2$ and for no $k$ is it
the case that $m_k=N-m_l$, then we replace $m_l$ by $N-m_l$.  The new
sequence of $m_l$ need no longer be increasing but that will be 
irrelevant; it suffices to prove the statement of the theorem with the
$F_l$ defined using these 
$m_l$.  It is still the case that all $m_l$ are distinct, and
we now have the property that if for some $l$ one has $m_l>N/2$, then
necessarily there is $k<l$ for which $N-m_l=m_k$.  

For convenience, let us set $m_0=0$ and enlarge the set of $m$'s to include
$m_0$.  Then $m_0=0$ and $m_r=N$ are both in
our enlarged set of $m$'s, and now the property stated above that if
$m_l>N/2$, then there is $k<l$ for which $N-m_l=m_k$ holds also for
$l=r$.  

Define polynomials $p_l$, $0\leq l\leq r$, as follows:
\[
\begin{array}{ll}
p_l=P_{m_l}\quad &  \text{   if  }\,\,\, m_l\leq N/2\\
p_l=Q_{m_l}\quad &  \text{   if  }\,\,\, m_l> N/2. 
\end{array}
\]
Clearly $\deg p_l\leq m_l$.

\bigskip

\noindent
{\bf Claim:}  There are constants $a_{j,l}$ for $0\leq l\leq r$, $0\leq 
j\leq l$, satisfying  
\begin{equation}\label{Fclaim}
F_l=\sum_{j=0}^l a_{j,l} p_j,\qquad 0\leq l\leq r
\end{equation}
\begin{equation}\label{a0induct}
a_{0,l}=\left[\prod_{j=1}^lm_j(N-m_j)\right]^{-1},\qquad 0\leq l\leq r-1
\end{equation}
\begin{equation}\label{arr}
a_{r,r}=\left[\prod_{l=1}^{r-1}m_l(N-m_l)\right]^{-1}.
\end{equation}

\bigskip
\noindent
In \eqref{a0induct} and \eqref{arr} an empty product is interpreted as 1. 

Proposition~\ref{recusolve} follows immediately from the Claim.  In fact,
all $p_j$ for 
$0\leq j\leq r-1$ have degree $<N$ and $p_r=Q_0$ has degree $=N$ by
Lemma~\ref{Qcoeff}.  Thus \eqref{Fclaim} for $l=r$ together with 
\eqref{arr} show that $F_r$ has
degree $=N$.  Only $p_r=Q_0$ contributes to its $y^N$ coefficient, which by
Lemma~\ref{Qcoeff} is $N^{-2}a_{r,r}$.  

The Claim is proved by induction on $l$.  It is clear for $l=0$ since 
$F_0=p_0=1$.
Suppose that the Claim is established for $l-1$ and assume first that  
$l<r$.  The argument is slightly different for the last induction step
passing from $l=r-1$ to $l=r$.  

Now $F_l$ is defined by
$$
\cD_{m_l}F_l=F_{l-1}=\sum_{j=0}^{l-1}a_{j,l-1}p_j,\qquad F_l(0)=0.  
$$
For
each $j$, $0\leq j\leq l-1$, we will solve $\cD_{m_l}u_j=p_j$, $u_j(0)=0$,
with $u_j$ a linear 
combination of the $p_k$, $0\leq k\leq l$.  Then
$F_l=\sum_{j=0}^{l-1}a_{j,l-1}u_j$ is of the desired form.  

The construction of the $u_j$'s is based on the observation 
\begin{equation}\label{Dshift}
\cD_{m_l}=\cD_{m_j} +[m_l(N-m_l)-m_j(N-m_j)].
\end{equation}
Consider different cases for $j$.  If $m_j\leq N/2$ and 
$m_j(N-m_j)\neq m_l(N-m_l)$, then $p_j=P_{m_j}$ solves $\cD_{m_j}p_j=0$.
Hence \eqref{Dshift} gives 
$$
\cD_{m_l}\left([m_l(N-m_l)-m_j(N-m_j)]^{-1}p_j\right)=p_j.
$$
Correct the value at $y=0$ by subtracting a multiple of the
solution of the homogeneous equation:  set 
$$
u_j=[m_l(N-m_l)-m_j(N-m_j)]^{-1}(p_j-P_{m_l}).
$$
Clearly $u_j$ solves the equation and the initial condition.  Now $P_{m_l}$
is of the form $p_k$ for some $k$ with $1\leq k\leq l$:  if $m_l\leq N/2$
then 
$P_{m_l}=p_l$, while if $m_l>N/2$, then $P_{m_l}=p_k$, where
$k<l$ is the index such that $N-m_l=m_k$.  Thus we have constructed $u_j$
of the desired form in this case.  Note that if $j=0$, then $m_j=0$ and our
solution is $u_0=[m_l(N-m_l)]^{-1}(p_0-P_{m_l})$.  The
coefficient of $p_0$ is $[m_l(N-m_l)]^{-1}$, and $p_0$ has
coefficient zero when any of the $u_j$ with $j>0$ is expressed as a linear
combination of the $p$'s.  

Next consider the construction of $u_j$ in case $m_j\leq N/2$ but 
$m_j(N-m_j)= m_l(N-m_l)$.  This case might not occur at all, and if it does
it can occur for only one $j$.  Since $j<l$ we have $m_j\neq m_l$, so it
must be that $m_l>N/2$ and $m_j=N-m_l$.  Therefore $p_j=P_{m_j}$
and $p_l=Q_{m_l}$.  Since $\cD_{m_l}Q_{m_l}=P_{m_l}=P_{m_j}$ and  
$Q_{m_l}(0)=0$, we just take $u_j=Q_{m_l}=p_l$.  $p_0$ does not occur in
the expression of this $u_j$ as a linear combination of the $p$'s.   

The remaining possibility is $m_j>N/2$.  Now we need to solve 
$\cD_{m_l}u_j=p_j=Q_{m_j}$.  Once again we apply \eqref{Dshift} to
conclude that 
\[
\begin{split}
\cD_{m_l}Q_{m_j}&=\cD_{m_j}Q_{m_j} +[m_l(N-m_l)-m_j(N-m_j)]Q_{m_j}\\
&=P_{m_j}+[m_l(N-m_l)-m_j(N-m_j)]Q_{m_j}.  
\end{split}
\]
Since $j<l$ it is impossible that $m_l=N-m_j$.  Therefore 
$m_l(N-m_l)-m_j(N-m_j)\neq 0$.  Arguing exactly as in the first case
above we conclude that we can solve 
$\cD_{m_l}v_j=P_{m_j}$, $v_j(0)=0$, with $v_j$ a linear combination of
the $p_k$ for $1\leq k\leq l$.  Then we take 
\[
\begin{split}
u_j&=[m_l(N-m_l)-m_j(N-m_j)]^{-1}(Q_{m_j}-v_j)\\
=&[m_l(N-m_l)-m_j(N-m_j)]^{-1}(p_j-v_j).
\end{split}
\]
Once again, $p_0$ has coefficient zero 
when $u_j$ is expressed as a linear combination of the $p$'s.    

This concludes the induction step for $l<r$: 
$F_l=\sum_{j=0}^{l-1}a_{j,l-1}u_j$ is of the desired form.  Since $p_0$
only entered in 
the construction of $u_0$, and its coefficient in $u_0$ was  
$[m_l(N-m_l)]^{-1}$, we have 
$$
a_{0,l}=[m_l(N-m_l)]^{-1}a_{0,l-1}.
$$
Thus \eqref{a0induct} follows by induction as well.

Finally consider the last inductive step, passing from $r-1$ to $r$.  
Now $m_l=N$ so $N-m_l=m_0=0$.  We again divide $\{j:0\leq j\leq r-1\}$ 
into the same three cases as above and solve for the $u_j$ using the same 
methods.  The difference now is that $j=0$ occurs in the second case  
instead of the first, since $m_0(N-m_0)=m_r(N-m_r)$.  So $u_0=Q_0=p_r$.   
In no other $u_j$ does $p_r$ occur with nonzero coefficient.  {From} 
$F_r=\sum_{j=0}^{r-1}a_{j,r-1}u_j$ we therefore deduce $a_{r,r}=a_{0,r-1}$,
which gives \eqref{arr}. 
\stopthm

This completes the proof of Theorem~\ref{combthm} and thus of
\eqref{beforechange}.  It remains to prove Proposition~\ref{modrecurcons}.   
It is evident upon expanding the $\cR_k$'s that the left-hand side of 
\eqref{Rform} is a linear combination of $\cM_{2I}(f)$.  Again make the 
change of variables \eqref{cov}.
Then \eqref{Rform} becomes 
\[
\begin{split}
\cL_{1-N}\cL_{3-N}\cdots \cL_{N-3}(s^{a-1})|_{s=0}  
&=\sum_{|I|=N-a}n_{(I,a)}(a-1)!^2(I_1-1)!^2\ldots (I_r-1)!^2\,\,x_I\\
&=\sum_{|I|=N-a}\nb_{(I,a)}\,\,x_I.
\end{split}
\]
But this is equivalent to \eqref{comb}, which stated
$$
\cL_{1-N}\cL_{3-N}\cdots \cL_{N-3}\cL_{N-1}1|_{s=0}  
=\sum_{|J|=N}\nb_J\,x_J
=\sum_{|(I,a)|=N}\nb_{(I,a)}\,\,x_Ix_a,  
$$
as one sees upon evaluating $\cL_{N-1}1=X(s)=\sum_{a\geq 1}x_as^{a-1}$.

\section{Recursive Formulae}\label{recursive}
In this section we present the proofs of Theorems~\ref{GJMSrecursive} and  
\ref{Qrecursive}.  First consider Theorem~\ref{GJMSrecursive}.  
Since $n_{(N)}=1$, \eqref{opexform} can  
be written as $P_{2N}=\cM_{2N}+\sum_{|I|=N, I\neq (N)}n_I\cM_{2I}$.  The
second term on the right-hand side only involves $\cM_{2M}$ with $M<N$.
Thus this is a polynomial lower-triangular system, and it follows  
that there are constants $a_I$ determined inductively by inverting this
relation so that 
$\cM_{2N}=P_{2N}+\sum_{|I|=N, I\neq (N)}a_IP_{2I}$.  Observe that
\eqref{oprecform} is another relation of this same form.  \S 2 of
\cite{J3} presents a proof due to Krattenthaler that  
\eqref{opexform} and \eqref{oprecform} are inverse relations in the 
other direction. Specifically, Krattenthaler showed that if    
$\overline{\cM}_{2N}$ are defined by
\begin{equation}\label{defMbar}
\overline{\cM}_{2N}=\sum_{|I|=N}m_IP_{2I}, 
\end{equation}
then 
\begin{equation}\label{Mbar}
P_{2N}=\sum_{|I|=N}n_I\overline{\cM}_{2I}.
\end{equation}
Our desired 
identity \eqref{oprecform} follows from the uniqueness of the inverse.
Concretely, from \eqref{Mbar} one deduces 
$\overline{\cM}_{2N}=P_{2N}+\sum_{|I|=N, I\neq (N)}a_IP_{2I}$ by precisely
the same inductive inversion as for the $\cM_{2N}$.  Hence 
$\overline{\cM}_{2N}=\cM_{2N}$, and \eqref{oprecform} follows.    

We review Krattenthaler's proof of \eqref{Mbar} as presented in \S 2 of
\cite{J3} as a warm-up for the proof of Theorem~\ref{Qrecursive}.   
Substitution of \eqref{defMbar} into \eqref{Mbar} shows that \eqref{Mbar}
is equivalent to 
\begin{equation}\label{Preduced}
P_{2N}=\sum_{|I|=N}\sum_{|J_1|=I_1,\ldots,|J_r|=I_r}
n_Im_{J_1}\cdots m_{J_r}P_{2J_1}\cdots P_{2J_r}. 
\end{equation}
The coefficient of $P_{2N}$ on the right-hand side is 1, so one is reduced
to showing that for $K=(K_1,\ldots,K_s)$ with $s>1$, the coefficient of
$P_{2K}$ in \eqref{Preduced} vanishes.  Given $K$, the choice 
of $J$'s corresponds to a choice of subset $A=\{a_1,\ldots,a_{r-1}\}$ of 
$[s-1]=\{1,\ldots,s-1\}$ (including the empty set) of cardinality
$r-1$, which we order by 
$1\leq a_1<a_2<\ldots<a_{r-1}\leq s-1$.  The parameterization is 
\begin{equation}\label{Js2}
\begin{split}
J_1=(K_{1},&\ldots,K_{a_1}),\quad J_2=(K_{a_1+1},\ldots,K_{a_2}),\;\ldots,\\
J_{r-1}&=(K_{a_{r-2}+1},\ldots,K_{a_{r-1}}), \quad
J_r=(K_{a_{r-1}+1},\ldots,K_s). 
\end{split}
\end{equation}
The $J$'s determine $I$ by $I=(|J_1|,\ldots, |J_r|)$.  
The coefficient of $P_{2K_1}\cdots P_{2K_s}$ 
is then 
\begin{equation}\label{Kcoeff2}
\sum_{A\subset [s-1]}n_{I}m_{J_1}\ldots m_{J_r},
\end{equation}
so \eqref{Mbar} reduces to showing that this vanishes for all 
$(K_1,\ldots,K_s)$ with $s>1$.  

Sums such as \eqref{Kcoeff2} can be evaluated using the following 
ingenious lemma of Krattenthaler.   
\begin{lemma}\label{Klemma}
Let $s>1$ and let $K_1,\ldots,K_s\in \N$.  Set $|K|=\sum_{j=1}^sK_j$.  For  
$A=\{a_1,\ldots,a_{r-1}\}\subset [s-1]$, define $J_1,\ldots, J_r$ and $I$
as above.  Then 
\begin{equation}\label{Kiden}
\begin{split}
\sum_{A\subset [s-1]}(-1)^rI_1\cdots I_{r-1}&(I_r+X)\cdot
\frac{\prod_{a\in A}(K_a+K_{a+1}+Y\chi(a=s-1))}
{\prod_{i=1}^{r-1}(\sum_{k=1}^iI_k)(\sum_{k=i+1}^rI_k)}\\ 
=&\;\frac{X(|K|-K_s)+Y(K_s+X)}{|K|-K_1}.
\end{split}
\end{equation}
Here $\chi(\mathcal{S})=1$ if $\mathcal{S}$ is true and
$\chi(\mathcal{S})=0$ otherwise.  $X$ and $Y$ are formal variables; the
identity holds as polynomials in $X$ and $Y$.  
\end{lemma}

\noindent
This is Lemma 2.1 in \cite{J3}.  The proof is by induction on $s$,
decomposing the set of subsets $A\subset [s]$ according to their last
element.  The proof is not at all obvious, but the real ingenuity was to
introduce 
the variables $X$ and $Y$ and to find the identity \eqref{Kiden} 
amenable to a proof by induction.  For 
the purposes of this paper it suffices to know \eqref{Kiden} in the case
$X=Y$.  An examination 
shows that the proof by induction actually applies to this case directly;
it is not necessary for our purposes to introduce both independent 
variables $X$ and $Y$.  We rewrite the identity for the case $X=Y$ in the
form we will need it in the proof of Theorem~\ref{Qrecursive}.  
Setting $X=Y=-b$ and replacing $K_s$ by $K_s+b$ in \eqref{Kiden} gives 
\begin{equation}\label{Kidenb}
%\begin{split}
\sum_{A\subset [s-1]}(-1)^rI_1\cdots I_r
\frac{\prod_{a\in A}(K_a+K_{a+1})}
{\prod_{i=1}^{r-1}(\sum_{k=1}^iI_k)(\sum_{k=i+1}^rI_k+b)}
=-\frac{b|K|}{|K|-K_1+b}.  
%\end{split}
\end{equation}
This holds also for $s=1$, since in that case both sides are $-K_1$.  As  
usual, empty products are interpreted as 1.  The form \eqref{Kidenb} seems
natural:  the induction hypothesis arises natually in its proof by
induction and the function $\chi(a=s-1)$ does not appear.   

We use Lemma~\ref{Klemma} to finish the proof of \eqref{Mbar}.  
Substitution of the definitions \eqref{nm} of $n_I$ and the $m_{J_i}$ into  
\eqref{Kcoeff2} shows that 
$$
\sum_{A\subset [s-1]}n_{I}m_{J_1}\ldots m_{J_r}=
(-1)^s(|K|-1)!^2\prod_{j=1}^s\frac{1}{K_j!(K_j-1)!}
\prod_{j=1}^{s-1}\frac{1}{K_j+K_{j+1}}\cdot \Sigma,  
$$
where $\Sigma$ is the expression occurring on the left-hand side of
\eqref{Kiden} with $X=Y=0$.  Lemma~\ref{Klemma} (or \eqref{Kidenb} with
$b=0$) shows that this vanishes. 
Thus \eqref{Mbar} follows, and hence also Theorem~\ref{GJMSrecursive}.      

We turn now to the proof of Theorem~\ref{Qrecursive}.  
Recall that the scalar invariants $W_{2N}$ are defined by \eqref{Wdef}. 
It will be convenient to introduce 
$$
\Wb_{2N}= 2^{2N}N!(N-1)!\,W_{2N},\qquad N\geq 1
$$
so that \eqref{Qexform} takes the form
\begin{equation}\label{Qexform2}
(-1)^NQ_{2N}=\sum_{|(I,b)|=N}n_{(I,b)}\cM_{2I}(\Wb_{2b})
\end{equation}
and \eqref{Qrecform} becomes
$$
\Wb_{2N}=\sum_{|(L,d)|=N}m_{(L,d)}(-1)^dP_{2L}(Q_{2d}).
$$
Substitution of \eqref{Qexform2} for each $(-1)^dQ_{2d}$ shows that
\eqref{Qrecform} is equivalent to
$$
\Wb_{2N}=\sum_{|(L,d)|=N}\sum_{|(I,b)|=d}m_{(L,d)}n_{(I,b)}P_{2L}\cM_{2I}(\Wb_{2b}).
$$
The term on the right-hand side with $L=I=\emptyset$ is $\Wb_{2N}$,  
so it suffices to prove 
$$
\sum_{|(L,d)|=N}\sum_{|(I,b)|=d}m_{(L,d)}n_{(I,b)}P_{2L}\cM_{2I}=0
$$
for each fixed $b$ such that $1\leq b<N$.  
Substitution of \eqref{oprecform} for each $\cM_{2I_j}$ rewrites this as 
\begin{equation}\label{bigsum}
\sum_{|(L,d)|=N}\sum_{|(I,b)|=d}
\sum_{|J_1|=I_1,\ldots,|J_r|=I_r}
m_{(L,d)}n_{(I,b)}m_{J_1}\cdots m_{J_r}P_{2L}P_{2J_1}\cdots P_{2J_r}=0.
\end{equation}

Fix $K_1,\ldots, K_s$ with $s\geq 1$ and each $K_j\geq 1$ and consider the 
coefficient of   
$P_{2K_1}\cdots P_{2K_s}$ in \eqref{bigsum}.  We must have 
$L=(K_1,\ldots, K_p)$ for some $p$, $0\leq p\leq s$.  Each $|J_i|\geq 1$, 
although $r=0$ is allowed corresponding to $p=s$.  For $p<s$, the 
choice of $J$'s corresponds to a choice of subset
$A=\{a_1,\ldots,a_{r-1}\}$ of  
$[s-p-1]=\{1,\ldots,s-p-1\}$ (including the empty set) of cardinality
$r-1$, which we order by 
$1\leq a_1<a_2<\ldots<a_{r-1}\leq s-p-1$.  Here
\begin{equation}\label{Js}
\begin{split}
J_1=(K_{p+1},&\ldots,K_{p+a_1}),\quad J_2=(K_{p+a_1+1},\ldots,K_{p+a_2}),\;\ldots,\\
J_{r-1}&=(K_{p+a_{r-2}+1},\ldots,K_{p+a_{r-1}}), \quad
J_r=(K_{p+a_{r-1}+1},\ldots,K_s). 
\end{split}
\end{equation}
For $p=s-1$, the only possibility for $A$ is the empty set, in which case  
$J_1=(K_s)$.  
The $J$'s determine $I$ by $I=(|J_1|,\ldots, |J_r|)$ as above.  
The coefficient of $P_{2K_1}\cdots P_{2K_s}$ is then
\begin{equation}\label{Kcoeff}
m_{(K,b)}+\sum_{p=0}^{s-1}m_{(L,|K|-|L|+b)}\sum_{A\subset
  [s-p-1]}n_{(I,b)}m_{J_1}\ldots m_{J_r}.
\end{equation}
So Theorem~\ref{Qrecursive} reduces to showing that this vanishes for all 
$b\geq 1$ and all $(K_1,\ldots,K_s)$ with $s\geq 1$.  

We use Lemma~\ref{Klemma} in the form \eqref{Kidenb} to evaluate the inner
sum.  Set  
$K'_j=K_{p+j}$ for $1\leq j\leq s-p$.  Substitution of  
\eqref{nm} for $n_{(I,b)}$ and the $m_{J_i}$ shows that
\begin{equation}\label{innersum}
\begin{split}
\sum_{A\subset [s-p-1]}n_{(I,b)}&m_{J_1}\ldots m_{J_r} \\
=&(-1)^{s-p}\frac{(|K'|+b-1)!^2}{|K'|\,b!\,(b-1)!}
\prod_{j=1}^{s-p}\frac{1}{K'_j!(K'_j-1)!} 
\prod_{j=1}^{s-p-1}\frac{1}{K'_j+K'_{j+1}}\cdot \Sigma
\end{split}
\end{equation}
where
$$
\Sigma=\sum_{A\subset [s-p-1]}(-1)^rI_1\cdots I_r
\frac{\prod_{a\in A}(K'_a+K'_{a+1})}
{\prod_{i=1}^{r-1}(\sum_{k=1}^iI_k)(\sum_{k=i+1}^rI_k+b)}. 
$$
Replacement of $s$ by $s-p$ and $K_j$ by $K_j'$ in \eqref{Kidenb} shows
that   
\begin{equation}\label{sigma}
\Sigma=-\frac{b|K'|}{|K'|-K_{p+1}+b}.  
\end{equation}

Substitute \eqref{sigma} into \eqref{innersum} and multiply by
$m_{(L,|K|-|L|+b)}$.  One obtains
\begin{equation}\label{msum}
\begin{split}
m_{(L,|K|-|L|+b)}&\sum_{A\subset [s-p-1]}n_{(I,b)}m_{J_1}\ldots m_{J_r} \\
=&(-1)^{s+1}\frac{(|K|+b)!(|K|+b-1)!}{(b-1)!^2}
\prod_{j=1}^s\frac{1}{K_j!(K_j-1)!} 
\prod_{j=1}^{s-1}\frac{1}{K_j+K_{j+1}}\cdot R_p
\end{split}
\end{equation}
where
$$
R_p=\frac{K_p+K_{p+1}}{(\sum_{i=p}^sK_i+b)(\sum_{i=p+1}^sK_i+b)(\sum_{i=p+2}^sK_i+b)}\;,
\qquad 1\leq p\leq s-1
$$
and
$$
R_0=\frac{1}{(\sum_{i=1}^sK_i+b)(\sum_{i=2}^sK_i+b)}.
$$
Empty sums are interpreted as 0.  

Set $b=K_{s+1}$ and substitute \eqref{msum} into \eqref{Kcoeff}.  After  
cancellation of factors in common with $m_{(K,b)}$, one finds that the 
vanishing of \eqref{Kcoeff} is equivalent to
$$
\sum_{p=1}^{s-1}\frac{K_p+K_{p+1}}
{(\sum_{i=p}^{s+1}K_i)(\sum_{i=p+1}^{s+1}K_i)(\sum_{i=p+2}^{s+1}K_i)} 
=\frac{1}{K_{s+1}(K_s+K_{s+1})}-\frac{1}{(\sum_{i=1}^{s+1}K_i)(\sum_{i=2}^{s+1}K_i)}. 
$$
This is proved by induction on $s$.  For $s=1$ the sum on the left-hand
side is empty and the right-hand side vanishes.  Suppose the identity holds
for $s$.  Write
\[
\begin{split}
\sum_{p=1}^{s}\frac{K_p+K_{p+1}}
{(\sum_{i=p}^{s+2}K_i)(\sum_{i=p+1}^{s+2}K_i)(\sum_{i=p+2}^{s+2}K_i)} 
=&\frac{K_1+K_2}
{(\sum_{i=1}^{s+2}K_i)(\sum_{i=2}^{s+2}K_i)(\sum_{i=3}^{s+2}K_i)} \\
+&\sum_{p=2}^s\frac{K_p+K_{p+1}}
{(\sum_{i=p}^{s+2}K_i)(\sum_{i=p+1}^{s+2}K_i)(\sum_{i=p+2}^{s+2}K_i)} 
\end{split}
\]
and use the induction hypothesis on the second term on the right-hand side
to obtain that the above equals 
\[
\begin{split}
&\frac{K_1+K_2}
{(\sum_{i=1}^{s+2}K_i)(\sum_{i=2}^{s+2}K_i)(\sum_{i=3}^{s+2}K_i)} 
+\frac{1}{K_{s+2}(K_{s+1}+K_{s+2})}-\frac{1}{(\sum_{i=2}^{s+2}K_i)(\sum_{i=3}^{s+2}K_i)}\\
&=\frac{1}{K_{s+2}(K_{s+1}+K_{s+2})}
+\frac{1}{(\sum_{i=2}^{s+2}K_i)(\sum_{i=3}^{s+2}K_i)}
\left(\frac{K_1+K_2}{\sum_{i=1}^{s+2}K_i}-1\right)\\
&=\frac{1}{K_{s+2}(K_{s+1}+K_{s+2})}
-\frac{1}{(\sum_{i=1}^{s+2}K_i)(\sum_{i=2}^{s+2}K_i)}.
\end{split}
\]
This completes the proof of the vanishing of \eqref{Kcoeff} and thus also 
of Theorem~\ref{Qrecursive}.

\end{document}